\documentclass[preprint,12pt]{elsarticle}

\usepackage{amsfonts}
\usepackage{hyperref}
\usepackage{mathrsfs}
\usepackage{indentfirst}
\usepackage{amsmath}
\usepackage{amssymb}
\usepackage{amsthm}
\usepackage{graphicx}
\usepackage{tikz}
\usepackage{float}
\usepackage{algorithm}
\usepackage{algorithmic}
\usepackage{lipsum}
\usepackage{color}

\newtheorem{lemma}{\bf Lemma}[section]
\newtheorem{theorem}{\bf Theorem}[section]
\newtheorem{remark}{\bf Remark}[section]

\newtheorem{example}{\bf Example}[section]

\newtheorem{proposition}{\bf Proposition}[section]

\makeatletter
\newenvironment{breakablealgorithm}
  {
   \begin{center}
     \refstepcounter{algorithm}
     \hrule height.8pt depth0pt \kern2pt
     \renewcommand{\caption}[2][\relax]{
       {\raggedright\textbf{\ALG@name~\thealgorithm} ##2\par}%
       \ifx\relax##1\relax 
         \addcontentsline{loa}{algorithm}{\protect\numberline{\thealgorithm}##2}%
       \else 
         \addcontentsline{loa}{algorithm}{\protect\numberline{\thealgorithm}##1}%
       \fi
       \kern2pt\hrule\kern2pt
     }
  }{
     \kern2pt\hrule\relax
   \end{center}
  }
\makeatother

\journal{}

\begin{document}

\begin{frontmatter}



\title{An algorithm for finding the scaled basis of the supereigenvector space in max-plus algebra}

\author{Hui-li Wang\corref{corw}}
\author{Serge{\u{\i}} Sergeev\corref{cors}}

\cortext[corw]{School of Sciences, Southwest Petroleum University, Chengdu, Sichuan 610500, China. Email: huiliwang77@163.com.\\
Supported by the National Natural Science Foundation of China (No.11901486), Youth Science and Technology Innovation Team of SWPU for Nonlinear Systems (No. 2017CXTD02).}

\cortext[cors]{University of Birmingham, School of Mathematics, Edgbaston B15 2TT, UK. Email: s.sergeev@bham.ac.uk.  Corresponding author.}

\begin{abstract}
We present an algorithm for finding a basis of the supereigenvector space in max-plus algebra.  The main ideas of the new algorithm are: finding
better generators by exploiting the main operation of the tropical double description method and making use of the previously known extremality criteria.
\end{abstract}

\end{frontmatter}

\section{Introduction}\label{intro}

Max-plus algebra is the linear algebra over max-plus semiring.  The latter semiring is the set $\overline{\mathbb{R}}:=\mathbb{R}\cup \{-\infty\}$ equipped with the operations $(\oplus, \otimes)$, where addition $a\oplus b=\max(a, b)$ and multiplication $a\otimes b=a+b$. Note that $-\infty$ is neutral for $\oplus$ and $0$ is neutral for $\otimes$.
Max-plus algebra has proved to be useful in a number of areas such as: automata
theory, scheduling theory and discrete event systems. See \cite{Baccelli1992,Peter2010,GondranMinoux2008,MPatWork,S.Kubo2018} for some
of these applications.

The operations $(\oplus, \otimes)$ are extended to matrices and vectors over $\overline{\mathbb{R}}$ in the usual way. Thus we have for matrices $A=(a_{ij})$ and $B=(b_{ij})$:
\begin{equation}
\begin{split}
& (A\oplus B)_{ij}=a_{ij}\oplus b_{ij},\quad (A\otimes B)_{ij}=\sum_{k}^{\oplus}a_{ik}\otimes
b_{kj}=\mbox{max}_{k}(a_{ik}+b_{kj}),\\
& (\lambda\otimes A)_{ij}=\lambda\otimes a_{ij}, \quad \lambda\in \overline{\mathbb{R}}
 \end{split}
 \end{equation}
 As usual, the first two operations require that $A$ and $B$ are of compatible sizes.

We denote by $I$ any square matrix whose diagonal entries are $0$ and off-diagonal ones are $-\infty$.
Throughout the paper we also denote by $-\infty$ any vector or matrix whose every component is $-\infty$, by abuse of notation.
Note that $-\infty$ is neutral for $\oplus$ and $I$ for $\otimes$. We also denote $N=\{1,2,\cdots, n\}$.

For $A\in \overline{\mathbb{R}}^{n\times n}$, the
problem of finding a vector
$\textbf{x}\in\overline{\mathbb{R}}^{n}$, $\textbf{x}\neq-\infty$ and a scalar
$\lambda\in\overline{\mathbb{R}}$ satisfying
\begin{eqnarray}\label{EqSup}
A\otimes\textbf{x}\geq\lambda\otimes\textbf{x}
\end{eqnarray}
is called the
(max-plus algebraic) supereigenproblem, the vector $\textbf{x}$ is called a
supereigenvector of $A$ associated with $\lambda$. If $\lambda=-\infty$, then $A\otimes\textbf{x}\geq\lambda\otimes\textbf{x}$ for any
$\textbf{x}\in\overline{\mathbb{R}}^{n}$, so we only need to consider the case of $\lambda\neq-\infty$ in \eqref{EqSup}. Note that in this case \eqref{EqSup} can be written as
\begin{eqnarray}\label{EqSupnolam}
A\otimes\textbf{x}\geq\textbf{x}
\end{eqnarray}
after multiplying both sides of \eqref{EqSup} by $\lambda^{-1}$ and rewriting $\lambda^{-1}\otimes A$ again as $A$. We will denote $\mathscr{X}=\{\textbf{x}\in\overline{\mathbb{R}}^{n}\colon
A\otimes\textbf{x}\geq\textbf{x}, \textbf{x}\neq-\infty\}$ and $\mathscr{X}'=\mathscr{X}\cup\{-\infty\}$ in the case of $\mathscr{X}\neq\emptyset$.  We call $\mathscr{X}'$ the supereigenspace of $A$ associated with $\lambda=0$. For $A\otimes \textbf{x}\geq \textbf{x}$ with $A\in \overline{\mathbb{R}}^{n\times n}$ and $\textbf{x}=(x_1, x_2, \cdots, x_n)^T\in\overline{\mathbb{R}}^{n}$, $A_{i}\otimes\textbf{x}\geq x_{i}$ stands for the $i$-inequality of the system $A\otimes \textbf{x}\geq \textbf{x}$, where $A_{i}$ denotes the $i$th row of $A$.

Note that the problem of finding all supereigenvectors \eqref{EqSupnolam} is in close relation to the problems of finding all max-plus subeigenvectors $\textbf{x}\colon A\otimes \textbf{x}\leq \textbf{x}$ and
all max-plus eigenvectors $\textbf{x}\colon A\otimes \textbf{x}=\textbf{x}$ (associated with $\lambda=0$.) However, these two problems are much easier to solve, see in particular the monograph of Butkovi\v{c} \cite{Peter2010}.

The problem of describing the supereigenvector space $\mathscr{X}'$ was posed by Butkovi\v{c}, Schneider and Sergeev \cite{Peter2012} and has attracted some attention since then.
It is a special case of a more general problem of describing the solution set of $A\otimes \textbf{x}\geq B\otimes \textbf{x}$ where $B=I$. A method for finding a set of generators for the solution set of $A\otimes \textbf{x}\geq B\otimes \textbf{x}$ was developed by Butkovi\v{c} and Heged\"{u}s~\cite{Peter1984}, and more recently, a much more efficient method was developed by Allamigeon, Gaubert and Goubault~\cite{Allamigeon2013} for finding all extremals of such solution set. However, as the system $A\otimes \textbf{x}\geq \textbf{x}$ is much more special, an original algorithm for finding all generators was suggested by Wang and Wang~\cite{wanghuili2014}, who also gave two algorithms for finding a finite proper supereigenvector for irreducible matrix and a finite solution of the system  \eqref{EqSupnolam} for reducible matrix \cite{wanghuili2015InformationSciences}. These algorithms are not in obvious relation with those of~\cite{Peter1984} and \cite{Allamigeon2013}. Based on the results in \cite{wanghuili2014}, Wang, Yang and Wang developed an improved algorithm which can be used to find a smaller set of generators \cite{wanghuili2020}. A set of generators for the solution set of \eqref{EqSupnolam} was also obtained by Krivulin~\cite{Nikolai2018} using an algebraic method called sparsification. A similar method was used by Krivulin~\cite{Nikolai2020} to obtain a generating set of $A\otimes \textbf{x}\geq B\otimes \textbf{x}$.

Butkovi\v{c} et al.~\cite{Peter2007} showed that any closed subspace of $\overline{\mathbb{R}}^n$ has the smallest scaled set of generators, which consists of scaled extremals of that subspace. Since the solution set of $A\otimes \textbf{x}\geq \textbf{x}$ is finitely generated and therefore closed, it is important to find the set of scaled extremals of it. 
Note that a general extremality criterion for a solution of $A\otimes \textbf{x}\geq B\otimes \textbf{x}$ was obtained by Allamigeon, Gaubert and Goubault~\cite{Allamigeon2013}. Being based on a notion of connectivity in hypergraphs, it requires $O(n^2)\times \alpha(n)$ where $\alpha(n)$ is the inverse of the Ackermann function. Sergeev and Wang \cite{SW-21} recently obtained a different set of criteria for the extremality of a solution of $A\otimes \textbf{x}\geq \textbf{x}$. The verification of their set of criteria does not require one to construct a hypergraph and has $O(n^2)$ complexity (thus not involving the Ackermann function in this special case).  A different approach was taken earlier by Sergeev~\cite{Sergeev2015} who  gave a combinatorial description of those generators of~\cite{wanghuili2014} that are extremal, as a set of criteria that can be verified in $O(n^2)$ time where $n$ is the dimension of $A$.  The criteria of~\cite{Sergeev2015} can be used to immediately decide about the extremality of any generator obtained by Wang and Wang~\cite{wanghuili2014}.

The purpose of this paper will be to develop an algorithm for finding all extremals of $\mathscr{X}'$, based on the earlier algorithms of Wang and Wang~\cite{wanghuili2014, wanghuili2020}. The resulting algorithm formally appears below as Algorithm~\ref{algorithm3}, but its essential parts are formulated earlier as Algorithm~\ref{algorithm1} and Algorithm~\ref{algorithm2}. Two main ideas of this development are: 1) use the main operation of the tropical double description method and any of the available extremality criteria to find a smaller number of generators with a more ``compact'' support, 2) prove that the resulting set of extremals actually generates the whole $\mathscr{X}'$ as all generators found by the algorithm of Wang and Wang~ \cite{wanghuili2020} (repeated here as Algorithm~\ref{algorithm4}) appear as
max-combinations of the ``new'' generators obtained in the present paper (Algorithm~\ref{algorithm3}). The latter fact is proved in Theorem~\ref{mainresult}.  Note that, although the main algorithm presented here uses the same main operation as the tropical double description method of Allamigeon, Gaubert and Goubault~\cite{Allamigeon2013}, it is essentially different, being more closely related to the algorithms of Wang and Wang~\cite{wanghuili2014, wanghuili2020} and Krivulin~\cite{Nikolai2018}.

The main algorithms and results of this paper (Algorithms~\ref{algorithm1}, \ref{algorithm2}, \ref{algorithm3} and Theorem~\ref{mainresult}) are presented in Section~\ref{s:main}.
Before this, we include Section~\ref{s:prel} explaining the main concepts and constructions of max-plus algebra that will be used.

\section{Preliminaries}
\label{s:prel}

Much of the techniques in the paper will be based on the connection between max-plus matrices and associated digraphs, which we next explain.
Given a square matrix $A=(a_{ij})\in \overline{\mathbb{R}}^{n\times n}$ the symbol
$D_{A}$ denotes the weighted directed graph $(N, E, w)$ with the node set $N$ and arc set
$E=\{(i, j): a_{ij}\neq-\infty\}$ and weight
$\omega(i,j) = a_{ij}$ for all $(i,j)\in E$.

A sequence of nodes $\pi=(i_{1},\cdots, i_{k})$ in $D_{A}$ is called a {\em path} (in $D_{A}$) if $(i_{j}, i_{j+1})\in E$ for all $j=1,\cdots,k-1$.
The length of such path $\pi$, denoted by $l(\pi)$, is $k-1$. The weight of such path $\pi$ is denoted by $\omega(\pi)$ and is equal to $a_{i_{1}i_{2}}+\cdots+a_{i_{k-1}i_{k}}$. If $\omega(\pi)\geq 0$ then $\pi$ is called a {\em nonnegative path.} Node $i_{1}$ is called the {\em starting node} and $i_{k}$ the {\em endnode} of $\pi$, respectively. A path from $i_{1}$ to $i_{k}$ can be called a $i_{1}-i_{k}$ path. Further $i_{k}$ is said to be {\em reachable} or {\em accessible} from $i_{1}$, notation $i_{1}\rightarrow i_{k}$.
If $i_{1}=i_{k}$ then the path $(i_{1},\cdots, i_{k})$ is called a {\em cycle}. Moreover, if $i_{p}\neq i_{q}$ for all $p, q=1,\cdots, k-1,
p\neq q$ then that cycle is called an {\em elementary cycle} (similarly, we can define {\em elementary path}). Cycle $(i_{1}, i_{2}=i_{1})$ is called a {\em loop}. 

The theory of generators, extremals and (weak) bases of subspaces in $\overline{\mathbb{R}}^{n}$ (also known as max cones or tropical cones) was described by Butkovi\v{c} et al.~\cite{Peter2007}, see also \cite{Peter2010}, and below we give some elements of it.

Let $S\subseteq\overline{\mathbb{R}}^{n}$. We say that $S$ is a {\em subspace} of $\mathbb{R}^n$ if $a\otimes\textbf{u}\oplus b\otimes\textbf{v}\in S$ for every $\textbf{u}, \textbf{v}\in S$ and $a, b\in \overline{\mathbb{R}}$.
A vector $\textbf{v}=(v_{1}, \cdots,
v_{n})^{T}\in\overline{\mathbb{R}}^{n}$ is called a {\em max-combination}
of $S$ if $\textbf{v}=\sum^{\oplus}_{\textbf{x}\in
S}\alpha_{\textbf{x}}\otimes \textbf{x}, \alpha_{\textbf{x}}\in
\overline{\mathbb{R}}$, where only a finite number of
$\alpha_{\textbf{x}}$ are real. Denote the set of all max-combinations of
$S$ by $\operatorname{span}(S)$. $S$ is called a set of {\em generators}
for $T$ if $\operatorname{span}(S)=T$. The set $S$ is called {\em (weakly) dependent} if $\textbf{v}$ is a
max-combination of $S\backslash\{\textbf{v}\}$ for some $\textbf{v}\in S$.
Otherwise $S$ is (weakly) independent. Let $S,
T\subseteq\overline{\mathbb{R}}^{n}$. If $S$ is an independent set of generators for $T$ then the set $S$ is called a basis
of $T$.

A vector $\textbf{v}\in S$ is called an {\em extremal} in $S$ if for all $\textbf{x}, \textbf{y}\in S$, $\textbf{v}=\textbf{x}\oplus \textbf{y}$ implies $\textbf{v}=\textbf{x}$ or $\textbf{v}=\textbf{y}$. A vector $\textbf{v}=(v_{1}, \cdots, v_{n})^{T}\in S$ ($\textbf{v}\neq -\infty$)is called {\em scaled} if $\|\textbf{v}\|=0$, where $\|\textbf{v}\|=\max_{i=1}^{n}v_{i}$.
The set $S$ is called scaled if all its elements are scaled. For vector $\textbf{v}=(v_{1}, \cdots, v_{n})^{T}\in \overline{\mathbb{R}}^{n}$ the support of $\textbf{v}$ is defined by
$$\text{Supp}(\textbf{v})=\{j\in N: v_{j}\in \mathbb{R}\}.$$

The following result explains the construction of a basis for a finitely generated subspaces of $\overline{\mathbb{R}}^n$.

\begin{proposition}[\bf Wagneur~\cite{Wagneur1991}, see also \cite{Peter2007,Peter2010}]
\label{Pro0.1}
If $T$ is a finitely generated subspace then its set of scaled extremals is nonempty and it is the unique scaled basis for $T$.
\end{proposition}

By Proposition \ref{Pro0.1}, the set of scaled extremals in $\mathscr{X}'$ is the unique scaled basis for $\mathscr{X}'$.
For a moment, let us consider a general system $B\otimes \textbf{x}\leq A\otimes \textbf{x}$
with $A, B\in \overline{\mathbb{R}}^{p\times d}$. Such system consists of $p$ inequalities $B_j\otimes \textbf{x}\leq A_j\otimes \textbf{x}$, where $A_j$ and $B_j$ are the $j$th rows of matrices $A$ and $B$.
Let us recall the double description method of \cite{Allamigeon2013}, which is based on the following claim.  Here and below, we denote by $e^{i}$ the vector of $\overline{\mathbb{R}}^{n}$ whose $i$th entry is $0$ and other ones are $-\infty$, for
$i=1,\ldots,n$. This is the $i$th tropical unit vector.
\begin{proposition}[\bf\cite{Allamigeon2013}]
\label{AllTh14}
Let $\mathscr{C}\subset\overline{\mathbb{R}}^{d}$ be a tropical polyhedral cone defined as the set $\{\emph{\textbf{x}}\in \overline{\mathbb{R}}^{d} \mid B\otimes \emph{\textbf{x}}\leq A\otimes \emph{\textbf{x}}\}$, where $A, B\in \overline{\mathbb{R}}^{p\times d}$ (with $p\geq 0$). Let $V_{0}, \cdots, V_{p}$ be the sequence of finite subsets of $\overline{\mathbb{R}}^{d}$ defined as follows:
\begin{equation}\label{EqAllTh14}
\begin{split}
  V_{0} = &\{e^{i}\}_{i\in [d]},\\
 V_k =  &\{v\in V_{k-1} \mid B_k\otimes v\leq A_k\otimes v\}\cup\\
  & \{(B_k\otimes w)\otimes v\oplus(A_k\otimes v)\otimes w\mid  v, w\in V_{k-1},\\
  &\qquad\qquad\qquad\qquad
   B_k\otimes v\leq A_k\otimes v,\  B_k\otimes w> A_k\otimes w\},
\end{split}
\end{equation}
for all $k=1, \cdots, p.$ Then $\mathscr{C}$ is generated by the finite set $V_p$.
\end{proposition}
Not aiming to use this method in a straightforward way, we will make use of the main operation on which it is based.
Let us denote $\mathscr{X}_i=\{\textbf{x}\colon B_i\otimes \textbf{x}\leq A_i\otimes \textbf{x}\}$ and include a proof of the following fact, which is implicit in~\cite{Allamigeon2013}.
\begin{lemma}
\label{l:operation}
Let $\emph{\textbf{x}}\in\mathscr{X}_i$ and $\emph{\textbf{y}}\in\overline{\mathbb{R}}^{d}$. Then
$$
\emph{\textbf{z}}=(B_i\otimes \emph{\textbf{y}})\otimes \emph{\textbf{x}}\oplus (A_i\otimes \emph{\textbf{x}})\otimes \emph{\textbf{y}}
$$
also belongs to $\mathscr{X}_i$.
\end{lemma}
\begin{proof}
Substituting $\textbf{z}$ into $B_i\otimes \textbf{z}\leq A_i\otimes \textbf{z}$ and using distributivity we obtain:
\begin{equation}
\label{e:operation}
\begin{split}
&(B_i\otimes \textbf{y})\otimes (B_i\otimes \textbf{x})\oplus (A_i\otimes \textbf{x})\otimes (B_i\otimes \textbf{y})\\
&\leq (B_i\otimes \textbf{y})\otimes (A_i\otimes \textbf{x})\oplus (A_i\otimes \textbf{x})\otimes (A_i\otimes \textbf{y})
\end{split}
\end{equation}
Since $(B_i\otimes \textbf{y})\otimes (B_i\otimes \textbf{x})\leq (B_i\otimes \textbf{y})\otimes (A_i\otimes \textbf{x})$ (using $B_i\otimes \textbf{x}\leq A_i\otimes \textbf{x}$)
and $(A_i\otimes \textbf{x})\otimes (B_i\otimes \textbf{y})=(B_i\otimes \textbf{y})\otimes (A_i\otimes \textbf{x})$, we obtain that the left hand side of~\eqref{e:operation} does not exceed the first term on the right hand side, hence~\eqref{e:operation} holds.
\end{proof}

System \eqref{EqSupnolam} can be treated as a special case of the more general system considered above, if we rewrite it as
\begin{eqnarray}\label{Foralgori}
B\otimes\textbf{x}\leq A\otimes\textbf{x},\quad \text{with $B=I$}.
\end{eqnarray}

\section{Algorithm for finding all extremal generators in $\mathscr{X}'$}
\label{s:main}

In this section we will develop an algorithm which can be used to find the scaled basis of $\mathscr{X}'$.

Recall that for the system (\ref{Foralgori}) $\mathscr{X}_{i}=\{\textbf{x}=(x_1, x_2, \cdots, x_n)^T\in\overline{\mathbb{R}}^{n}:
A_{i}\otimes\textbf{x}\geq x_{i}\}$. In the rest of this section, unless stated otherwise, for $\textbf{x}\in\overline{\mathbb{R}}^{n}$ we will denote by $(\textbf{x})_{j}$ the $j$-th entry of $\textbf{x}$.

\begin{lemma}
\label{Lemmaforgenera2}
Take any $i\in N$. Then $e^{i}\in\mathscr{X}_{j}$ for any $j\in N\backslash\{i\}$.
\end{lemma}
\begin{proof}
For any $j\in N, j\neq i$, the $j$-th entry of $e^{i}$ is $-\infty$. The claim then follows since $A_{j}\otimes e^{i}\geq -\infty$.
\end{proof}

\begin{lemma}
\label{Lemmaforgenera}
Let $\emph{\textbf{x}}=(x_1, x_2, \cdots, x_n)^T\in \overline{\mathbb{R}}^{n}$. If $\emph{\textbf{x}}\in \mathscr{X}_{i}$ for any $i\in \emph{\text{Supp}}(\emph{\textbf{x}})$, then $\emph{\textbf{x}}\in \mathscr{X}$.
\end{lemma}
\begin{proof}
If $\text{Supp}(\textbf{x})=N$, then we have that  $A_{i}\otimes\textbf{x}\geq x_{i}$ for all $i\in N$ and the statement follows.
If $\text{Supp}(\textbf{x})\subsetneqq N$, then $x_i=-\infty$ for any $i\in N\backslash\text{Supp}(\textbf{x})$ and
$A_i\otimes\textbf{x}\geq x_i=-\infty$ for all such $i$ as well, thus again $A_{i}\otimes\textbf{x}\geq x_{i}$ for all $i$.
\end{proof}

It is known that the system (\ref{Foralgori}) is solvable if and only if $\lambda(A)\geq0$ \cite{wanghuili2014}. We will therefore assume in the rest of this paper that $\lambda(A)\geq0$ for the system (\ref{Foralgori}).
Given nonnegative elementary cycle $\sigma=(i_{1}, i_{2}, \cdots, i_{t}, i_{t+1}=i_{1})$ in $D_{A}$, we denote $J_{\sigma}=\{j\in N: j\mbox{ } \text{in}\mbox{ }\sigma\}$. With different nodes in $J_{\sigma}$ as the starting nodes, the cycle $\sigma$ can be rewritten as $(j_{1}=i_{k}, j_{2}, \cdots, j_{t+1}=j_{1})$ for some $1\leq k\leq t$ and will be denoted by $\sigma(i_{k})$.

We will first show how Lemmas~\ref{l:operation},  \ref{Lemmaforgenera2} and \ref{Lemmaforgenera} can be used for finding one solution to the system \eqref{Foralgori} corresponding to a nonnegative elementary cycle in $D_{A}$.

Consider a nonnegative elementary cycle $\sigma=\sigma(i_{k})=(j_{1}=i_{k}, j_{2}, \cdots, j_{t+1}=j_{1})$ in $D_{A}$ and let $k$ $(1\leq k\leq t)$ be fixed. We now show that a solution to system \eqref{Foralgori} can be found by iterating at most $t$ times. Set $v^{1}=e^{j_{1}}$. If $v^{1}\in\mathscr{X}_{j_{1}}$, then $v^{1}\in\mathscr{X}$ by Lemma \ref{Lemmaforgenera} and the iteration stops. If
$v^{1}\notin \mathscr{X}_{j_{1}}$, then we set
$$v^{2}: =(B_{j_{1}}\otimes v^{1})\otimes e^{j_{2}}\oplus (A_{j_{1}}\otimes e^{j_{2}})\otimes v^{1}=
e^{j_{2}}\oplus (A_{j_{1}}\otimes e^{j_{2}})\otimes v^{1}.$$
The only finite components of this vector are $(v^2)_{j_1}=a_{j_{1}j_{2}}$ and $(v^2)_{j_2}=0$.

By Lemma~\ref{Lemmaforgenera2} we have $e^{j_{2}}\in \mathscr{X}_{j_{1}}$, and by
Lemma~\ref{l:operation} we have $v^{2}\in\mathscr{X}_{j_{1}}$.
If $v^{2}\in\mathscr{X}_{j_{2}}$, then Lemma \ref{Lemmaforgenera} implies that $v^{2}\in\mathscr{X}$ since $\text{Supp}(v^{2})=\{j_{1}, j_{2}\}$, and the iteration stops. If $v^{2}\notin\mathscr{X}_{j_{2}}$,
then we can set
\begin{equation*}
v^{3}:  =(B_{j_{2}}\otimes v^{2})\otimes e^{j_{3}}\oplus (A_{j_{2}}\otimes e^{j_{3}})\otimes v^{2}
=e^{j_{3}}\oplus (A_{j_{2}}\otimes e^{j_{3}})\otimes v^{2}.
\end{equation*}
The only finite components of this vector are $(v^3)_{j_1}=a_{j_{1}j_{2}}+a_{j_{2}j_{3}}$, $(v^3)_{j_2}=a_{j_{2}j_{3}}$ and $(v^3)_{j_3}=0$.

By Lemma \ref{Lemmaforgenera2} $e^{j_{3}}\in \mathscr{X}_{j_{1}}$ and $e^{j_{3}}\in \mathscr{X}_{j_{2}}$. Then we have $v^3\in \mathscr{X}_{j_{1}}$ since $e^{j_3}, v^2\in \mathscr{X}_{j_{1}}$.  Lemma~\ref{l:operation} implies that
$v^{3}\in\mathscr{X}_{j_{2}}$.
Next we continue to determine whether $v^{3}\in\mathscr{X}_{j_{3}}$ and similarly whether $v^p\in\mathscr{X}_{j_p}$ for any $p\leq t-1$.  More precisely, by $t$th iteration of this procedure we will have obtained vector $v^{1}=e^{j_{1}}\notin\mathscr{X}_{j_{1}}$ and (if $t\geq 3$) vectors $v^p$ for $2\leq p\leq t-1$  such that $v^p\in\mathscr{X}_{j_l}$ for $l\colon 1\leq l\leq p-1$, but $v^p\notin \mathscr{X}_{j_p}$. We then define
\begin{equation}
\label{e:vt}
\begin{split}
v^t: =(B_{j_{t-1}}\otimes v^{t-1})\otimes e^{j_{t}}\oplus (A_{j_{t-1}}\otimes e^{j_{t}})\otimes v^{t-1}
=e^{j_{t}}\oplus (A_{j_{t-1}}\otimes e^{j_{t}})\otimes v^{t-1}
\end{split}
\end{equation}
Using simple induction we can obtain the following finite components of $v^t$:
\begin{equation}
\label{e:vtcomps}
(v^t)_{j_t}=0,\quad (v^t)_{j_l}=a_{j_{l}j_{l+1}}+\ldots + a_{j_{t-1}j_{t}},\quad l=1,\ldots t-1.
\end{equation}
Since $e^{j_t}, v^{t-1}\in\mathscr{X}_{j_l}$ for $l\colon 1\leq l\leq t-2$, we obtain $v^t\in\mathscr{X}_{j_l}$ for $l\colon 1\leq l\leq t-2$.
Lemma~\ref{l:operation} implies that $v^t\in\mathscr{X}_{j_{t-1}}$.
We now also claim that $v^{t}\in\mathscr{X}_{j_{t}}$.
By~\eqref{e:vtcomps} we have
$(v^{t})_{j_{1}}=a_{j_{1}j_{2}}+a_{j_{2}j_{3}}+\cdots+a_{j_{t-1}j_{t}}$ and
$(v^{t})_{j_{t}}=0$. Thus
$$A_{j_{t}}\otimes v^{t}\geq a_{j_{t}j_{1}}+(v^{t})_{j_{1}}=a_{j_{t}j_{1}}+a_{j_{1}j_{2}}+a_{j_{2}j_{3}}+\cdots+a_{j_{t-1}j_{t}}=\omega(\sigma)\geq 0=(v^{t})_{j_{t}},$$
as claimed. Also observe that $\text{Supp}(v^{t})=J_{\sigma}$, and then by Lemma \ref{Lemmaforgenera} we have $v^{t}\in \mathscr{X}$. From the above discussion, for a nonnegative elementary cycle $\sigma=\sigma(i_{k})=(j_{1}=i_{k}, j_{2}, \cdots, j_{t+1}=j_{1})$ in $D_{A}$ we can always find a solution, say $v^{\sigma(i_{k})}$, to the system (\ref{Foralgori}) through no more than $t$ iterations.

Note that the criteria of Allamigeon et al.~\cite[Theorem1]{Allamigeon2013} or Sergeev and Wang~\cite[Theorem 3.1]{SW-21}
can then be used to perform the extremality test for $v^{\sigma(i_{k})}$ and eliminate non-extremal solution. In this way we can obtain at most $t$ extremal solutions to the system (\ref{Foralgori}) by taking all $1\leq k\leq t$ and denote the set of these extremal solutions (corresponding to nonnegative elementary cycle $\sigma=(i_{1}, i_{2}, \cdots, i_{t}, i_{t+1}=i_{1})$ in $D_{A}$) by $\mathscr{X}^{\sigma}$.

Based on the above discussion, we summarize the following algorithm to find the set of extremal solutions $\mathscr{X}^{\sigma}$ corresponding to the nonnegative elementary cycle $\sigma=(i_{1}, i_{2}, \cdots, i_{t}, i_{t+1}=i_{1})$ in $D_{A}$.

\begin{breakablealgorithm}
 \renewcommand{\algorithmicrequire}{\textbf{Input:}}
 \renewcommand{\algorithmicensure}{\textbf{Output:}}
 \caption{ \label{algorithm1}}
 \begin{algorithmic}[1]
 \REQUIRE  A nonnegative elementary cycle $\sigma=(i_{1}, i_{2}, \cdots, i_{t}, i_{t+1}=i_{1})$ in $D_{A}$.
 \STATE Set $\mathscr{X}^{\sigma}:=\emptyset$ and $k=1$. 
 \WHILE{$k\leq t$}
 \STATE Let $\sigma(i_{k})=(j_{1}=i_{k}, j_{2}, \cdots, j_{t+1}=j_{1})$.
 \STATE Set $v^{p}:=e^{j_{p}}$ and proceed with $p=1$.
 \WHILE{$p\leq t-1$ and $v^{p}\notin \mathscr{X}_{j_{p}}$}
 \STATE $v^{p+1}=e^{j_{p+1}}\oplus (A_{j_{p}}\otimes e^{j_{p+1}})\otimes v^{p}$.
 \STATE $p=p+1$.
 \ENDWHILE
 \STATE Set $v^{\sigma(i_{k})}: =v^{p}$ and do extremality test for $v^{\sigma(i_{k})}$.
 \IF {$v^{\sigma(i_{k})}$ is an extremal}
 \STATE $\mathscr{X}^{\sigma}:=\mathscr{X}^{\sigma}\cup\{\parallel v^{\sigma(i_{k})}\parallel^{-1}\otimes v^{\sigma(i_{k})}\}$.
 \ENDIF
 \STATE $k=k+1$.
 \ENDWHILE
 \ENSURE $\mathscr{X}^{\sigma}$.
 \end{algorithmic}
\end{breakablealgorithm}

Given nonnegative elementary cycle $\sigma=(i_{1}, i_{2}, \cdots, i_{t}, i_{t+1}=i_{1})$ in $D_{A}$. Recall that
a path $(l_{1}, l_{2},
\cdots, l_{m})$ is called a maximum $J_{\sigma}$-path \cite{wanghuili2014} if the following three conditions are satisfied:
(1) $l_{p}\neq l_{q}$
for any $1\leq p\neq q\leq m$,
(2) the endnode $l_{m}$ is the
unique node which belongs to $J_{\sigma}$ in the path,
(3) there are no paths from $l$ 
to $l_{1}$ in $D_{A}$ for any $l\in N\backslash (J_{\sigma}\cup\{l_{1}, l_{2}, \cdots,
l_{m}\})$. Next we show how to efficiently find some extremal solutions to the system (\ref{Foralgori}) by using Lemma~\ref{l:operation} corresponding to some maximum $J_{\sigma}$-path. We start with a simple but useful statement:

\begin{lemma}
\label{Lemmaforgenera3}
Let $\emph{\textbf{x}}=(x_1, x_2, \cdots, x_n)^T\in \mathscr{X}$ with $\emph{\text{Supp}}(\emph{\textbf{x}})\subsetneqq N$ be such that $a_{ji}\neq -\infty$ for some $j\in N\backslash\emph{\text{Supp}}(\emph{\textbf{x}})$ and $i\in \emph{\text{Supp}}(\emph{\textbf{x}})$, and consider $\emph{\textbf{y}}:= \emph{\textbf{x}}\oplus (A_j\otimes \emph{\textbf{x}})\otimes e^j$.  Then
\begin{itemize}
\item[{\rm (i)}]  $\emph{\textbf{y}}\in \mathscr{X}$;
\item[{\rm (ii)}] If $\emph{\textbf{x}}$ is not an extremal of $\mathscr{X}$ or if $e^j\in\mathscr{X}_j$, then $\emph{\textbf{y}}$ is not an extremal
of $\mathscr{X}$.
\end{itemize}
\end{lemma}
\begin{proof}
(i): Observing that $\textbf{y}=(B_j\otimes e^j)\otimes \textbf{x}\oplus (A_j\otimes \textbf{x})\otimes e^j,$ we first establish that
$\textbf{y}\in\mathscr{X}_i$ for $i\in\text{Supp}(\textbf{x})$ (since $\textbf{x}\in\mathscr{X}_i$ and $e^j\in\mathscr{X}_i$)
and then apply Lemma~\ref{l:operation} to establish that $\textbf{y}\in\mathscr{X}_j$.

(ii): If $e^j\in\mathscr{X}_j$ then $e^j\in\mathscr{X}$ and, since the supports of $\textbf{x}$ and $e^j$ are disjoint, this implies that $\textbf{y}$ is not an extremal. If $\textbf{x}$ is not an extremal, then we can write it as $\textbf{x}=\textbf{x}_1\oplus \textbf{x}_2$
where both $\textbf{x}^1$ and $\textbf{x}^2$ are different from $\textbf{x}$ and belong to $\mathscr{X}$.
Then
\begin{align*}
\textbf{y}&=\textbf{x}\oplus(A_{j}\otimes\textbf{x})\otimes e^{j}\\
&=\textbf{x}^{1}\oplus\textbf{x}^{2}\oplus
(A_{j}\otimes (\textbf{x}^{1}\oplus\textbf{x}^{2}))\otimes e^{j}\\
&=\textbf{x}^{1}\oplus \textbf{x}^{2}
\oplus(A_{j}\otimes\textbf{x}^{1})\otimes e^{j}\oplus(A_{j}\otimes\textbf{x}^{2})\otimes e^{j}
\end{align*}
Denote
\begin{align*}
 \textbf{y}^{1}=\textbf{x}^{1}\oplus(A_{j}\otimes\textbf{x}^{1})\otimes e^j
 =(B_{j}\otimes e^{j})\otimes\textbf{x}^{1}\oplus(A_{j}\otimes\textbf{x}^{1})\otimes e^{j},\\
 \textbf{y}^{2}=\textbf{x}^{2}\oplus(A_{j}\otimes\textbf{x}^{2})\otimes e^j
 =(B_{j}\otimes e^{j})\otimes\textbf{x}^{2}\oplus(A_{j}\otimes\textbf{x}^{2})\otimes e^j
 \end{align*}
 Then $\textbf{y}=\textbf{y}^{1}\oplus\textbf{y}^{2}$.   As in part (i), using Lemma~\ref{l:operation} we establish that $\textbf{y}^{1}\in \mathscr{X}$ and $\textbf{y}^{2}\in\mathscr{X}$. Also, $\textbf{x}_1\neq\textbf{x}$ and $\textbf{x}_2\neq\textbf{x}$ obviously imply (in this case) that
$\textbf{y}_1\neq\textbf{y}$ and $\textbf{y}_2\neq\textbf{y}$.
\end{proof}

\if{
If there exist $j\in N-\text{Supp}(\textbf{x})$ such that
$a_{ji}\neq-\infty$ for some $i\in \text{Supp}(\textbf{x})$, then take any such $j$
and denote $\textbf{y}=(B_{j}\otimes\epsilon^{j})\otimes\textbf{x}\oplus(A_{j}\otimes\textbf{x})\otimes\epsilon^{j}=(y_1, y_2, \cdots, y_n)^T$.
Since $a_{ji}\neq-\infty$ for some $i\in \text{Supp}(\textbf{x})$ we have that $A_{j}\otimes\textbf{x}$ is finite. Suppose that $A_{j}\otimes\textbf{x}=a_{ji'}+x_{i'}$ with some $i'\in \text{Supp}(\textbf{x})$.
Then the construction of vector $\textbf{y}$ (recall that $B=I$) implies that $y_{k}=x_{k}$ for all $k\neq j$, $k\in N$ and $y_{j}=a_{ji'}+x_{i'}$ (since $x_{j}=-\infty$). Thus it is sufficient to prove that $A_{j}\otimes\textbf{y}\geq y_{j}$. In fact, $A_{j}\otimes\textbf{y}\geq a_{ji'}+y_{i'}=a_{ji'}+x_{i'}= y_{j}$. Hence the statement.
}\fi
\if{
\begin{remark}\hspace{-0.6em}\textbf{.}\label{Remark1}
For convenience we still use all the symbols in the proof of Lemma \ref{Lemmaforgenera3}.
Take one fixed $j$ satisfying the condition in Lemma \ref{Lemmaforgenera3}. If $\emph{\textbf{x}}$ is not an extremal or $\epsilon^{j}\in \mathscr{X}_{j}$, then $\emph{\textbf{y}}=(B_{j}\otimes\epsilon^{j})\otimes\emph{\textbf{x}}\oplus(A_{j}\otimes\emph{\textbf{x}})\otimes\epsilon^{j}$ is not an extremal.
In the first case there exist $\emph{\textbf{x}}^{1}, \emph{\textbf{x}}^{2}\in\mathscr{X}^{\leq\emph{\textbf{x}}}$ such that $\emph{\textbf{x}}=\emph{\textbf{x}}^{1}\oplus\emph{\textbf{x}}^{2}$. Then
\begin{align*}
\emph{\textbf{y}}&=(B_{j}\otimes\epsilon^{j})\otimes\emph{\textbf{x}}\oplus(A_{j}\otimes\emph{\textbf{x}})\otimes\epsilon^{j}\\
&=(B_{j}\otimes\epsilon^{j})\otimes(\emph{\textbf{x}}^{1}\oplus\emph{\textbf{x}}^{2})\oplus(A_{j}\otimes
(\emph{\textbf{x}}^{1}\oplus\emph{\textbf{x}}^{2}))\otimes\epsilon^{j}\\
&=(B_{j}\otimes\epsilon^{j})\otimes\emph{\textbf{x}}^{1}\oplus(B_{j}\otimes\epsilon^{j})\otimes\emph{\textbf{x}}^{2}
\oplus(A_{j}\otimes\emph{\textbf{x}}^{1})\otimes\epsilon^{j}\oplus(A_{j}\otimes\emph{\textbf{x}}^{2})\otimes\epsilon^{j}
\end{align*}
Denote $\emph{\textbf{y}}^{1}=(B_{j}\otimes\epsilon^{j})\otimes\emph{\textbf{x}}^{1}\oplus(A_{j}\otimes\emph{\textbf{x}}^{1})\otimes\epsilon^{j}$ and $\emph{\textbf{y}}^{2}=(B_{j}\otimes\epsilon^{j})\otimes\emph{\textbf{x}}^{2}\oplus(A_{j}\otimes\emph{\textbf{x}}^{2})\otimes\epsilon^{j}$. Then $\emph{\textbf{y}}=\emph{\textbf{y}}^{1}\oplus\emph{\textbf{y}}^{2}$. Obviously, we have $\emph{\textbf{y}}^{1}\in \mathscr{X}$ and
$\emph{\textbf{y}}^{2}\in\mathscr{X}$. Furthermore, according to the discussion on the construction of \emph{\textbf{y}} (similarly for $\emph{\textbf{y}}^{1}$ and $\emph{\textbf{y}}^{2}$) in the proof of Lemma \ref{Lemmaforgenera3}, we can see that $\emph{\textbf{y}}\neq\emph{\textbf{y}}^{1}$ and $\emph{\textbf{y}}\neq\emph{\textbf{y}}^{2}$ since $\emph{\textbf{x}}\neq\emph{\textbf{x}}^{1}$ and $\emph{\textbf{x}}\neq\emph{\textbf{x}}^{2}$. Thus $\emph{\textbf{y}}$ is not an extremal. In the second case, that is, $\epsilon^{j}\in \mathscr{X}_{j}$,  we have $\epsilon^{j}\in \mathscr{X}$ by Lemma \ref{Lemmaforgenera}. This also means that $\emph{\textbf{y}}$ is not an extremal.
\end{remark}
}\fi

Given a maximum $J_{\sigma}$-path $\pi=(l_{1}, l_{2},
\cdots, l_{m})$ in $D_{A}$, where nonnegative elementary cycle $\sigma=(i_{1}, i_{2}, \cdots, i_{t}, i_{t+1}=i_{1})$ in $D_{A}$ and $l_{m}=i_{k}$ for some $1\leq k\leq t$,
we can obtain one solution $v^{\sigma(i_{k})}$ to \eqref{Foralgori} for $\sigma=\sigma(i_{k})=(j_{1}=i_{k}, j_{2}, \cdots, j_{t+1}=j_{1})$ by using Algorithm \ref{algorithm1}. Obviously such $v^{\sigma(i_{k})}$ and $l_{m-1}$ satisfy the conditions of Lemma~\ref{Lemmaforgenera3}, which implies that
$v^{\sigma(i_{k})}\oplus(A_{l_{m-1}}\otimes v^{\sigma(i_{k})})\otimes e^{l_{m-1}}\in \mathscr{X}$. Continue to work with
$v^{\sigma(i_{k})}\oplus(A_{l_{m-1}}\otimes v^{\sigma(i_{k})})\otimes e^{l_{m-1}}$ and $l_{m-2}$ and so on. On the other hand, in the above process of constructing solution to the system \eqref{Foralgori}, once one of the situations described in
Lemma~\ref{Lemmaforgenera3} part (ii) occurs, the process stops. For every such constructed solution we can perform an extremality test and denote by $\mathscr{X}^{\pi-\sigma}$ the extremal solutions set obtained by using above mentioned method corresponding to maximum $J_{\sigma}$-path $\pi$ in $D_{A}$.

Formally, this process can be described by the following algorithm:

\begin{breakablealgorithm}
 \renewcommand{\algorithmicrequire}{\textbf{Input:}}
 \renewcommand{\algorithmicensure}{\textbf{Output:}}
 \caption{ \label{algorithm2}}
 \begin{algorithmic}[1]
 \REQUIRE  A maximum $J_{\sigma}$-path $\pi=(l_{1}, l_{2}, \cdots, l_{m})$ in $D_{A}$, where $\sigma=(i_{1}, i_{2}, \cdots, i_{t}, i_{t+1}=i_{1})$ is a nonnegative elementary cycle in $D_{A}$ and $l_{m}=i_{k}$ for some $1\leq k\leq t$ and $v^{\sigma(i_k)}$ is an extremal returned by Algorithm 1.
 \STATE Set $\mathscr{X}^{\pi-\sigma}: =\emptyset$.  Let $v^{l_{m}}=v^{\sigma(i_{k})}$ and proceed with $q=1$.
 \WHILE{$q\leq m-1$}
 \IF{$e^{l_{m-q}}\notin \mathscr{X}_{l_{m-q}}$}
 \STATE $v^{l_{m-q}}= v^{l_{m-q+1}}\oplus (A_{l_{m-q}}\otimes v^{l_{m-q+1}})\otimes e^{l_{m-q}}$.
 \STATE Do extremality test for $v^{l_{m-q}}$.
 \IF {$v^{l_{m-q}}$ is an extremal}
 \STATE $\mathscr{X}^{\pi-\sigma}:=\mathscr{X}^{\pi-\sigma}\cup\{\parallel v^{l_{m-q}}\parallel^{-1}\otimes v^{l_{m-q}}\}$.
 \ELSE
 \STATE 
 return $\mathscr{X}^{\pi-\sigma}$.
 \ENDIF
 \ELSE
 \STATE 
 return $\mathscr{X}^{\pi-\sigma}$.
 \ENDIF
 \STATE $q=q+1$.
 \ENDWHILE
 \ENSURE $\mathscr{X}^{\pi-\sigma}$.
 \end{algorithmic}
\end{breakablealgorithm}


When we implement Algorithms \ref{algorithm1} and \ref{algorithm2} for all nonnegative elementary cycles and all maximum $J_{\sigma}$-paths in $D_{A}$, we can obtain a set of scaled extremal solutions, denoted by $\mathscr{X}_{E}$, which can be computed by the following algorithm:

\begin{breakablealgorithm}
 \renewcommand{\algorithmicrequire}{\textbf{Input:}}
 \renewcommand{\algorithmicensure}{\textbf{Output:}}
 \caption{ \label{algorithm3}}
 \begin{algorithmic}[1]
 \REQUIRE  $A\in\overline{\mathbb{R}}^{n\times n}$.
 \STATE Let $\zeta=\{\sigma: \sigma \mbox{ is a nonnegative elementary
cycle in } D_{A}\}$ and $\mathscr{X}_{E}:=\emptyset$.
 \WHILE {$\zeta\neq\emptyset$}
 \STATE Take
$\sigma=(i_{1},\cdots,i_{t},i_{t+1}=i_{1})\in\zeta$.
 \STATE Apply Algorithm \ref{algorithm1} to $\sigma$ and set $\mathscr{X}_{E}: =\mathscr{X}_{E}\cup\mathscr{X}^{\sigma}$.
 \STATE Set $\vartheta=\{\pi: \pi \mbox{ is a maximum }J_{\sigma}\mbox{-}path\mbox{ in } D_{A}\}$.
 \WHILE {$\vartheta\neq\emptyset$}
 \STATE Take $\pi=(l_{1}, l_{2}, \cdots, l_{m})\in\vartheta$ with $l_{m}=i_{k}$ for some $k$ $(1\leq k\leq t)$.
 \IF{$v^{\sigma(i_k)}$ is an extremal}
 \STATE Apply Algorithm \ref{algorithm2} to $\pi$ and set $\mathscr{X}_{E}: =\mathscr{X}_{E}\cup\mathscr{X}^{\pi-\sigma}$.
 \ENDIF
 \STATE $\vartheta: =\vartheta\backslash\{\pi\}$.
 \ENDWHILE
 \STATE $\zeta: =\zeta\backslash\{\sigma\}$.
 \ENDWHILE
 \ENSURE $\mathscr{X}_{E}$.
 \end{algorithmic}
\end{breakablealgorithm}


If the steps of extremality test in Algorithms \ref{algorithm1} and \ref{algorithm2} are omitted, then we can obtain a set of solutions to the system (\ref{Foralgori}) by Algorithm \ref{algorithm3} and denoted it by $\mathscr{X}_{G}$. Denote $\mathscr{X}_{SG}=\{\parallel \textbf{x}\parallel^{-1}\otimes \textbf{x}: \textbf{x}\in\mathscr{X}_{G}\}$. Obviously $\mathscr{X}_{E}$ consists precisely of vectors of $\mathscr{X}_{SG}$ which are extremals of $\mathscr{X}'$.
In fact, $\mathscr{X}_{E}$ is also the scaled extremal set of $\mathscr{X}'$, that is, $\mathscr{X}_{E}$ is the scaled basis of $\mathscr{X}'$.

We recall two algorithms in \cite{wanghuili2020} that are needed for proof of the next result, Theorem \ref{mainresult}. Note that the following algorithm is obtained by combining Algorithm 3.1 and Algorithm 4.1 in \cite{wanghuili2020}, where $\mathscr{X}_{A}$ denote a generating set for $\mathscr{X}'$ and $\textbf{x}[j; a]$ denote the vector in
which the $j$th entry is $a$ and all other components are
$-\infty$.

\begin{breakablealgorithm}
\renewcommand{\algorithmicrequire}{\textbf{Input:}}
\renewcommand{\algorithmicensure}{\textbf{Output:}}
 \caption{ \label{algorithm4}}
 \begin{algorithmic}[1]
 \REQUIRE  $A\in\overline{\mathbb{R}}^{n\times n}$.
 \STATE Let $\zeta=\{\sigma: \sigma \mbox{ is a nonnegative elementary
cycle in } D_{A}\}$ and $\mathscr{X}_{A}:=\emptyset$.
 \WHILE {$\zeta\neq\emptyset$}
 \STATE Take
$\sigma=(i_{1},\cdots,i_{t},i_{t+1}=i_{1})\in\zeta$.
 \STATE Set $\textbf{x}_{j}=((\textbf{x}_{j})_1,
(\textbf{x}_{j})_2, \cdots,
(\textbf{x}_{j})_n)^T=(-\infty, -\infty, \cdots,
-\infty)^T$ for all $1\leq j\leq t$.
 \STATE Compute $\omega(\sigma)$ and proceed with $j=1$.
 \WHILE {$j\leq t$}
 \STATE Set $(\textbf{x}_{j})_{i_{1}}:=0$, and let
$\alpha_{j}=\omega(\sigma)$ and
$\alpha_{l}=0$ for all $1\leq l\leq t$ and $l\neq j$, and proceed with $s=1$.
 \WHILE {$s\leq t-1$}
 \STATE $(\textbf{x}_{j})_{i_{s+1}}:=(\textbf{x}_{j})_{i_{s}}+\alpha_{s}-a_{i_{s}i_{s+1}}$.
 \STATE $s=s+1$.
 \ENDWHILE
 \STATE Set $\mathscr{X}_{A}:=\mathscr{X}_{A}\cup\{\textbf{x}_{j}\}$.
 \STATE $j=j+1$.
 \ENDWHILE
 \STATE Set $\vartheta=\{\pi: \pi \mbox{ is a maximum }J_{\sigma}\mbox{-}path\mbox{ in } D_{A}\}$.
 \WHILE {$\vartheta\neq\emptyset$}
 \STATE Take $\pi=(l_{1}, l_{2}, \cdots, l_{m})\in\vartheta$ with $l_{m}=i_{k}$ for some $k$ $(2\leq k\leq t+1)$.
 \STATE $\textbf{x}:=\textbf{x}_{k-1}$, $c:=(\textbf{x}_{k-1})_{i_{k}}$ and proceed with $p=1$.
 \WHILE {$p\leq m-1$}
 \STATE $c :=c+a_{l_{m-p}l_{m-p+1}}$, $\textbf{x} :=\textbf{x}\oplus\textbf{x}[l_{m-p}; c]$.
 \STATE Set $\mathscr{X}_{A}:=\mathscr{X}_{A}\cup\{\textbf{x}\}$.
 \STATE $p=p+1$.
 \ENDWHILE
 \STATE $\vartheta: =\vartheta\backslash\{\pi\}$.
 \ENDWHILE
 \STATE $\zeta: =\zeta\backslash\{\sigma\}$.
 \ENDWHILE
 \ENSURE $\mathscr{X}_{A}$.
 \end{algorithmic}
\end{breakablealgorithm}

\begin{theorem}
\label{mainresult}
Let $A\in\overline{\mathbb{R}}^{n\times n}$ with $\lambda(A)\geq0$. Then $\mathscr{X}_{E}$ is the scaled basis of $\mathscr{X}'$.
\end{theorem}

\begin{proof}
Due to the fact that a set of generators for $\mathscr{X}'$ has been described in \cite{wanghuili2014}, see also \cite{Sergeev2015}, and subsequently a smaller generating set (denoted by $\mathscr{X}_{A}$) for $\mathscr{X}'$ is characterized in \cite{wanghuili2020}, and $\mathscr{X}_{E}$ is the scaled extremal set of $\mathscr{X}_{SG}$, it is sufficient to prove that any generator in $\mathscr{X}_{A}$ is a max-combination
of $\mathscr{X}_{SG}$ or, equivalently, $\mathscr{X}_{G}$.

Note that any generator $\textbf{x}\in \mathscr{X}_{A}$ corresponds to either a nonnegative elementary cycle $\sigma$ in $D_{A}$ or a maximum $J_{\sigma'}$-path $\pi=(l_{1}, l_{2},
\cdots, l_{m})$ in $D_{A}$, where $\sigma'$ is nonnegative elementary cycle in $D_{A}$.

Take any nonnegative elementary cycle $\sigma=(i_{1}, i_{2}, \cdots, i_{t}, i_{t+1}=i_{1})$ in $D_{A}$. Then by Algorithm \ref{algorithm4} we can get $t$ generators (say) $\textbf{x}_{1}, \textbf{x}_{2},\cdots, \textbf{x}_{t}$. By the similarity of the construction for generators $\textbf{x}_{1}, \textbf{x}_{2},\cdots, \textbf{x}_{t}$ we only need to prove that $\textbf{x}_{1}$ is a max-combination
of $\mathscr{X}_{G}$, where $\textbf{x}_{1}=((\textbf{x}_{1})_1,
(\textbf{x}_{1})_2, \cdots, (\textbf{x}_{1})_n)^T\in\overline{\mathbb{R}}^{n}$ can be computed by Algorithm \ref{algorithm4} as follows:

\begin{equation}\label{huili2020}
\begin{aligned}
&(\textbf{x}_{1})_{i_{1}}=0\\
&(\textbf{x}_{1})_{i_{2}}=\omega(\sigma)-a_{i_{1}i_{2}}\\
&(\textbf{x}_{1})_{i_{3}}=\omega(\sigma)-a_{i_{1}i_{2}}-a_{i_{2}i_{3}}\\
&\vdots\\
&(\textbf{x}_{1})_{i_{p}}=\omega(\sigma)-a_{i_{1}i_{2}}-\cdots-a_{i_{p-1}i_{p}}\\
&\vdots\\
&(\textbf{x}_{1})_{i_{t}}=\omega(\sigma)-a_{i_{1}i_{2}}-\cdots-a_{i_{t-1}i_{t}}\\
&(\textbf{x}_{1})_{s}=-\infty \mbox{ }\text{for any}\mbox{ } s\in N\backslash J_{\sigma}
\end{aligned}
\end{equation}

When we implement Algorithm \ref{algorithm1} for $\sigma=\sigma(i_{2})=(i_{2}, i_{3}, \cdots, i_{t}, i_{1}, i_{2})$, it is assumed that $v^{\sigma(i_{2})}$ is obtained by iterating $p$ $(1\leq p\leq t)$ times.

If $p=t$, then by  Algorithm \ref{algorithm1} we can get:
\begin{equation}\label{sigma(i_{2})p=t}
\begin{aligned}
&(v^{\sigma(i_{2})})_{i_{2}}=a_{i_{2}i_{3}}+\cdots+a_{i_{t}i_{1}}\\
&(v^{\sigma(2)})_{i_{3}}=a_{i_{3}i_{4}}+\cdots+a_{i_{t}i_{1}}\\
&\vdots\\
&(v^{\sigma(i_{2})})_{i_{t}}=a_{i_{t}i_{1}}\\
&(v^{\sigma(i_{2})})_{i_{1}}=0\\
&(v^{\sigma(i_{2})})_{s}=-\infty \mbox{ }\text{for any}\mbox{ } s\in N\backslash J_{\sigma}.
\end{aligned}
\end{equation}
By \eqref{huili2020} and \eqref{sigma(i_{2})p=t} we have $\textbf{x}_{1}=v^{\sigma(i_{2})}$,that is, $\textbf{x}_{1}$ can be expressed as a max-combination of $\mathscr{X}_{G}$.

If $p=1$, then $v^{\sigma(i_{2})}=e^{i_{2}}$. If $1<p<t$, then by  Algorithm \ref{algorithm1} we can get:
\begin{equation}\label{sigma(i_{2})p<t}
\begin{aligned}
&(v^{\sigma(i_{2})})_{i_{2}}=a_{i_{2}i_{3}}+\cdots+a_{i_{p}i_{p+1}}\\
&(v^{\sigma(i_{2})})_{i_{3}}=a_{i_{3}i_{4}}+\cdots+a_{i_{p}i_{p+1}}\\
&\vdots\\
&(v^{\sigma(i_{2})})_{i_{p}}=a_{i_{p}i_{p+1}}\\
&(v^{\sigma(i_{2})})_{i_{p+1}}=0\\
&(v^{\sigma(i_{2})})_{s}=-\infty \mbox{ }\text{for any}\mbox{ } s\in N\backslash\{i_{2}, i_{3}, \cdots, i_{p+1}\}.
\end{aligned}
\end{equation}

We further assume that $v^{\sigma(i_{p+2})}$ is obtained by using Algorithm \ref{algorithm1} for $\sigma=\sigma(i_{p+2})=(i_{p+2}, \cdots, i_{t}, i_{1}, \cdots, i_{p+1},i_{p+2})$ in the case of $p<t-1$ or $\sigma=\sigma(i_{1})=(i_{1}, \cdots, i_{t}, i_{1})$ in the case of $p=t-1$, and iterating $q$ $(1\leq q\leq t)$ times in both cases. We only need to prove for $1<p<t-1$ ($p=1$ and $p=t-1$ can be proved similarly). There are three cases:

Case 1. If $p+q=t$, then by  Algorithm \ref{algorithm1} we have:
\begin{equation}\label{sigma(i_{p+2})}
\begin{aligned}
&(v^{\sigma(i_{p+2})})_{i_{p+2}}=a_{i_{p+2}i_{p+3}}+\cdots+a_{i_{t}i_{1}}\\
&\vdots\\
&(v^{\sigma(i_{p+2})})_{i_{t}}=a_{i_{t}i_{1}}\\
&(v^{\sigma(i_{p+2})})_{i_{1}}=0\\
&(v^{\sigma(i_{p+2})})_{s}=-\infty \mbox{ }\text{for any}\mbox{ } s\in N\backslash\{i_{p+2}, i_{p+3}, \cdots, i_{t}, i_{1}\}.
\end{aligned}
\end{equation}
Using (\ref{huili2020}), (\ref{sigma(i_{2})p<t}) and (\ref{sigma(i_{p+2})}) we obtain that $$\textbf{x}_{1}=v^{\sigma(i_{2})}\otimes(a_{i_{p+1}i_{p+2}}+\cdots+a_{i_{t-1}i_{t}}+a_{i_{t}i_{1}})\oplus v^{\sigma(i_{p+2})},$$
that is, $\textbf{x}_{1}$ can be expressed as a max-combination of $\mathscr{X}_{G}$.

Case 2. If $p+q>t$ (obviously $p+q< 2t$), then by  Algorithm \ref{algorithm1} we have:
\begin{equation}\label{sigma(i_{p+2}more)}
\begin{aligned}
&(v^{\sigma(i_{p+2})})_{i_{p+2}}=a_{i_{p+2}i_{p+3}}+\cdots+a_{i_{t-1}i_{t}}+a_{i_{t}i_{1}}+a_{i_{1}i_{2}}+\cdots+a_{i_{p+q-t}i_{p+q-t+1}}\\
&\vdots\\
&(v^{\sigma(i_{p+2})})_{i_{t}}=a_{i_{t}i_{1}}+a_{i_{1}i_{2}}+\cdots+a_{i_{p+q-t}i_{p+q-t+1}}\\
&(v^{\sigma(i_{p+2})})_{i_{1}}=a_{i_{1}i_{2}}+\cdots+a_{i_{p+q-t}i_{p+q-t+1}}\\
&\vdots\\
&(v^{\sigma(i_{p+2})})_{i_{p+q-t}}=a_{i_{p+q-t}i_{p+q-t+1}}\\
&(v^{\sigma(i_{p+2})})_{i_{p+q-t+1}}=0\\
&(v^{\sigma(i_{p+2})})_{s}=-\infty \mbox{ }\text{for any}\mbox{ } s\in N\backslash\{i_{p+2}, \cdots, i_{t}, i_{1},\cdots, i_{p+q-t+1}\}.
\end{aligned}
\end{equation}
Using \eqref{huili2020}, \eqref{sigma(i_{2})p<t}, \eqref{sigma(i_{p+2}more)} and that $w(\sigma)\geq 0$
we obtain that
\begin{equation}\label{sigma(i_{p+2}more1)}
\textbf{x}_{1}=v^{\sigma(i_{2})}\otimes(a_{i_{p+1}i_{p+2}}+\cdots+a_{i_{t-1}i_{t}}+a_{i_{t}i_{1}})\oplus v^{\sigma(i_{p+2})}\otimes(-a_{i_{1}i_{2}}-\cdots-a_{i_{p+q-t}i_{p+q-t+1}}),
\end{equation}
that is, $\textbf{x}_{1}$ can be expressed as a max-combination of $\mathscr{X}_{G}$.

Case 3. If $p+q<t$, then by  Algorithm \ref{algorithm1} we have:

\begin{equation}\label{sigma(p+2morep+q<t-2)}
\begin{aligned}
&(v^{\sigma(i_{p+2})})_{i_{p+2}}=a_{i_{p+2}i_{p+3}}+\cdots+a_{i_{p+q}i_{p+q+1}}\\
&(v^{\sigma(i_{p+2})})_{i_{p+3}}=a_{i_{p+3}i_{p+4}}+\cdots+a_{i_{p+q}i_{p+q+1}}\\
&\vdots\\
&(v^{\sigma(i_{p+2})})_{i_{p+q}}=a_{i_{p+q}i_{p+q+1}}\\
&(v^{\sigma(i_{p+2})})_{i_{p+q+1}}=0\\
&(v^{\sigma(i_{p+2})})_{s}=-\infty \mbox{ }\text{for any}\mbox{ } s\in N\backslash\{i_{p+2}, i_{p+3}, \cdots, i_{p+q+1}\}.
\end{aligned}
\end{equation}

If we continue in this way by considering $v^{\sigma(i_{p+q+2})}$ $\cdots$ then we can always find a finite number of solutions in $\mathscr{X}_{G}$ such that the support of such solutions can cover the support of $\textbf{x}_{1}$. Further $\textbf{x}_{1}$ can be expressed as a max-combination of $\mathscr{X}_{G}$ similarly as in Case 1 and Case 2.  More precisely, it can be shown that in this case and in general we have
\begin{equation}\label{general}
\textbf{x}_{1}=v^{\sigma(i_{2})}\otimes c_{0}\oplus z,
\end{equation}
where $c_{0}$ is a real number satisfying $(\textbf{x}_{1})_{i_{2}}=(v^{\sigma(i_{2})})_{i_{2}}\otimes c_{0}$ and $z$ is a max-combination of $\mathscr{X}_{G}$.

Take any maximum $J_{\sigma'}$-path $\pi=(l_{1}, l_{2},
\cdots, l_{m})$ in $D_{A}$, where $\sigma'$ is nonnegative elementary cycle in $D_{A}$. Then obviously, $m\geq 2$.
Without loss of generality assume that $\sigma'=\sigma=(i_{1}, i_{2}, \cdots, i_{t}, i_{t+1}=i_{1})$ and $l_{m}=i_{2}$. Then by Algorithm \ref{algorithm4} we can get $m-1$ generators corresponding to $\pi$, say, $\textbf{x}_{1}(l_{m-1}),
\textbf{x}_{1}(l_{m-2}), \cdots, \textbf{x}_{1}(l_{1})$, where $\textbf{x}_{1}(l_{m-1})$ satisfying $\text{Supp}(\textbf{x}_{1}(l_{m-1}))=J_{\sigma}\cup\{l_{m-1}\}$ and $(\textbf{x}_{1}(l_{m-1}))_{r}=(\textbf{x}_{1})_{r}$ for any $r\in N$, $r\neq l_{m-1}$ and $(\textbf{x}_{1}(l_{m-1}))_{l_{m-1}}=(\textbf{x}_{1})_{i_{2}}+a_{l_{m-1}i_{2}}$,
$\textbf{x}_{1}(l_{m-2})$ satisfying $\text{Supp}(\textbf{x}_{1}(l_{m-2}))=\text{Supp}(\textbf{x}_{1}(l_{m-1}))\cup\{l_{m-2}\}$ and $(\textbf{x}_{1}(l_{m-2}))_{r}=(\textbf{x}_{1}(l_{m-1}))_{r}$ for any $r\in N$, $r\neq l_{m-2}$ and $(\textbf{x}_{1}(l_{m-2}))_{l_{m-2}}=(\textbf{x}_{1}(l_{m-1}))_{l_{m-1}}+a_{l_{m-2}l_{m-1}}=(\textbf{x}_{1})_{i_{2}}+a_{l_{m-1}i_{2}}+a_{l_{m-2}l_{m-1}}$,
$\cdots$,
$\textbf{x}_{1}(l_{1})$ satisfying $\text{Supp}(\textbf{x}_{1}(l_{1}))=\text{Supp}(\textbf{x}_{1}(l_{2}))\cup\{l_{1}\}$ and $(\textbf{x}_{1}(l_{1}))_{r}=(\textbf{x}_{1}(l_{2}))_{r}$ for any $r\in N$, $r\neq l_{1}$ and $(\textbf{x}_{1}(l_{1}))_{l_{1}}=(\textbf{x}_{1}(l_{2}))_{l_{2}}+a_{l_{1}l_{2}}=\cdots=
(\textbf{x}_{1})_{i_{2}}+a_{l_{m-1}i_{2}}+a_{l_{m-2}l_{m-1}}+\cdots+a_{l_{1}l_{2}}$.

Next we show that each $\textbf{x}_{1}(l_{m-k})$ $(1\leq k\leq m-1)$ can be expressed as a max-combination of $\mathscr{X}_{G}$.

By the previous discussion, there exist some real number $c_{0}$ satisfying $(\textbf{x}_{1})_{i_{2}}=(v^{\sigma(i_{2})})_{i_{2}}\otimes c_{0}$ and a max-combination of $\mathscr{X}_{G}$, say $z$, such that $\textbf{x}_{1}$ can be expressed as in \eqref{general}.

For $k=1$, if $A_{l_{m-1}}\otimes (v^{\sigma(i_{2})}\otimes c_{0})=a_{l_{m-1}i_{2}}+(\textbf{x}_{1})_{i_{2}}$, that is $A_{l_{m-1}}\otimes (v^{\sigma(i_{2})}\otimes c_{0})=({\textbf{x}}_{1}(l_{m-1}))_{l_{m-1}}$, then it is straightforward to verify that
\begin{equation}\label{maxpath2}
{\textbf{x}}_{1}(l_{m-1})=v^{l_{m-1}}\otimes c_{0}\oplus\textbf{x}_{1},
\end{equation}
where
\begin{equation}\label{maxpath3}
v^{l_{m-1}}=v^{\sigma(i_{2})}\oplus(A_{l_{m-1}}\otimes v^{\sigma(i_{2})})\otimes e^{l_{m-1}}.
\end{equation}
It follows from \eqref{general} and \eqref{maxpath2} that ${\textbf{x}}_{1}(l_{m-1})$ can be expressed as a max-combination of $\mathscr{X}_{G}$.

If $A_{l_{m-1}}\otimes (v^{\sigma(i_{2})}\otimes c_{0})\neq a_{l_{m-1}i_{2}}+(\textbf{x}_{1})_{i_{2}}$, that is, $A_{l_{m-1}}\otimes (v^{\sigma(i_{2})}\otimes c_{0})>({\textbf{x}}_{1}(l_{m-1}))_{l_{m-1}}$.  Note that, since $(\textbf{x}_{1})_{i_{2}}=(v^{\sigma(i_{2})})_{i_{2}}\otimes c_{0}$ and $(\textbf{x}_{1}(l_{m-1}))_{l_{m-1}}=(\textbf{x}_{1})_{i_{2}}+a_{l_{m-1}i_{2}}$,
the inequality $A_{l_{m-1}}\otimes (v^{\sigma(i_{2})}\otimes c_{0})<({\textbf{x}}_{1}(l_{m-1}))_{l_{m-1}}$ is impossible.

Set
\begin{equation}\label{maxpath10}
a=A_{l_{m-1}}\otimes (v^{\sigma(i_{2})}\otimes c_{0})-({\textbf{x}}_{1}(l_{m-1}))_{l_{m-1}}.
\end{equation}
Then we have
\begin{equation}\label{maxpath4}
{\textbf{x}}_{1}(l_{m-1})=v^{l_{m-1}}\otimes (c_{0}-a)\oplus{\textbf{x}}_{1},
\end{equation}
where $v^{l_{m-1}}$ is also defined by (\ref{maxpath3}).
It follows from (\ref{general}) and (\ref{maxpath4}) that ${\textbf{x}}_{1}(l_{m-1})$ can be expressed as a max-combination of $\mathscr{X}_{G}$.

By (\ref{maxpath2}) and (\ref{maxpath4}) we can get that, in general,
\begin{equation}\label{general1}
{\textbf{x}}_{1}(l_{m-1})=v^{l_{m-1}}\otimes c_{1}\oplus{\textbf{x}}_{1}=v^{l_{m-1}}\otimes c_{1}\oplus v^{\sigma(i_{2})}\otimes c_{0}\oplus z,
\end{equation}
where $c_{1}$ is a real number satisfying $({\textbf{x}}_{1}(l_{m-1}))_{l_{m-1}}=(v^{l_{m-1}})_{l_{m-1}}\otimes c_{1}$, $c_{0}$ is a real number satisfying $(\textbf{x}_{1})_{i_{2}}=(v^{\sigma(i_{2})})_{i_{2}}\otimes c_{0}$ and $z$ is a max-combination of $\mathscr{X}_{G}$.

For $k=2$, if $A_{l_{m-2}}\otimes (v^{l_{m-1}}\otimes c_{1})=a_{l_{m-2}l_{m-1}}+({\textbf{x}}_{1}(l_{m-1}))_{l_{m-1}}=({\textbf{x}}_{1}(l_{m-2}))_{l_{m-2}}$, then
\begin{equation}\label{maxpath5}
{\textbf{x}}_{1}(l_{m-2})=v^{l_{m-2}}\otimes c_{1}\oplus{\textbf{x}}_{1}(l_{m-1}),
\end{equation}
where
\begin{equation}\label{maxpath6}
v^{l_{m-2}}=v^{l_{m-1}}\oplus(A_{l_{m-2}}\otimes v^{l_{m-1}})\otimes e^{l_{m-2}}.
\end{equation}
It follows from (\ref{general1}) and (\ref{maxpath5}) that ${\textbf{x}}_{1}(l_{m-2})$ can be expressed as a max-combination of $\mathscr{X}_{G}$.

If $A_{l_{m-2}}\otimes (v^{l_{m-1}}\otimes c_{1})>a_{l_{m-2}l_{m-1}}+({\textbf{x}}_{1}(l_{m-1}))_{l_{m-1}}=({\textbf{x}}_{1}(l_{m-2}))_{l_{m-2}}$. Set $b=A_{l_{m-2}}\otimes (v^{l_{m-1}}\otimes c_{1})-(a_{l_{m-2}l_{m-1}}+({\textbf{x}}_{1}(l_{m-1}))_{l_{m-1}})=A_{l_{m-2}}\otimes (v^{l_{m-1}}\otimes c_{1})-({\textbf{x}}_{1}(l_{m-2}))_{l_{m-2}}$. Then
\begin{equation}\label{maxpath8}
{\textbf{x}}_{1}(l_{m-2})=v^{l_{m-2}}\otimes (c_{1}-b)\oplus{\textbf{x}}_{1}(l_{m-1}),
\end{equation}
where $v^{l_{m-2}}$ is also defined by (\ref{maxpath6}).
It follows from (\ref{general1}) and (\ref{maxpath8}) that ${\textbf{x}}_{1}(l_{m-2})$ can be expressed as a max-combination of $\mathscr{X}_{G}$.

By (\ref{maxpath5}) and (\ref{maxpath8}) we can get that, in general,
\begin{equation}\label{general2}
{\textbf{x}}_{1}(l_{m-2})=v^{l_{m-2}}\otimes c_{2}\oplus {\textbf{x}}_{1}(l_{m-1})=v^{l_{m-2}}\otimes c_{2}\oplus v^{l_{m-1}}\otimes c_{1}\oplus v^{\sigma(i_{2})}\otimes c_{0}\oplus z,
\end{equation}
where $c_{k}$ is a real number satisfying $({\textbf{x}}_{1}(l_{m-k}))_{l_{m-k}}=(v^{l_{m-k}})_{l_{m-k}}\otimes c_{k}$ for $k=1, 2$, $c_{0}$ is a real number satisfying $(\textbf{x}_{1})_{i_{2}}=(v^{\sigma(i_{2})})_{i_{2}}\otimes c_{0}$ and $z$ is a max-combination of $\mathscr{X}_{G}$.

If we continue in this way then we can obtain successively for any $k=3, \cdots, m-1$ (if $m\geq 4$):
\begin{equation}\label{general3}
{\textbf{x}}_{1}(l_{m-k})=\sum^{\oplus}_{1\leq h\leq k} v^{l_{m-h}}\otimes c_{h}\oplus v^{\sigma(i_{2})}\otimes c_{0}\oplus z,
\end{equation}
where $c_{h}$ is a real number satisfying $({\textbf{x}}_{1}(l_{m-h}))_{l_{m-h}}=(v^{l_{m-h}})_{l_{m-h}}\otimes c_{h}$ for any $1\leq h\leq k$, $c_{0}$ is a real number satisfying $(\textbf{x}_{1})_{i_{2}}=(v^{\sigma(i_{2})})_{i_{2}}\otimes c_{0}$ and $z$ is a max-combination of $\mathscr{X}_{G}$.


Thus any generator in the generating set $\mathscr{X}_{A}$ for $\mathscr{X}'$ is a max-combination
of $\mathscr{X}_{G}$ (equivalently $\mathscr{X}_{SG}$), together with $\mathscr{X}_{E}$ is the scaled extremal set of $\mathscr{X}_{SG}$, yields that $\mathscr{X}_{E}$ is the scaled basis of $\mathscr{X}'$.
\end{proof}

Let us now give an example illustrating the work of Algorithms~\ref{algorithm1}, \ref{algorithm2} and \ref{algorithm3}.

\begin{example}
\label{Ex4.1}
{\rm Consider the system $A\otimes \textbf{x}\geq\textbf{x}$, where
$$A=\left(\begin{array}{ccccc}
            -3 & 1 & -\infty & -\infty & -\infty \\
            1 & 1 & 1 & -\infty & -\infty \\
            -\infty & 0 & -\infty & 2 & -\infty \\
            1 & -\infty & -5 & -\infty & -7 \\
            -2 & -2 & -7 & 1 & -\infty
          \end{array}
   \right).$$

By using the tropical double description method of Allamigeon, Gaubert and Goubault \cite{Allamigeon2013} or \cite[Algorithm 3.2]{wanghuili2014} or \cite[Algorithm 4.1]{wanghuili2020} we can get that the scaled basis of $\mathscr{X}'$ is
\begin{equation*}
\begin{split}
& \{e^{2},\  (0, -1, -\infty, -\infty, -\infty)^{T},\  (-\infty, 0, 0, -\infty, -\infty)^{T},\\
& (-\infty, 0, 0, -5, -\infty)^{T},\  (-3, -4, 0, -2, -\infty)^{T},\  (-\infty, 0, -\infty, -\infty, -2)^{T},\\
& (0, -1, -\infty, -\infty, -2)^{T},\  (-\infty, 0, -\infty, -9, -2)^{T},\\
&\qquad\qquad  (-1, -2, -\infty, 0, -\infty)^{T},\  (-2, -3, -\infty, -1, 0)^{T}\}.
\end{split}
\end{equation*}

Next we use our method to obtain these extremals. There are four nonnegative elementary cycles in $D_{A}$: $\sigma_{1}=(2, 2)$, $\sigma_{2}=(2, 1, 2)$, $\sigma_{3}=(2, 3, 2)$, $\sigma_{4}=(2, 3, 4, 1, 2)$. For the nonnegative elementary cycle $\sigma_{1}$ we have the following maximum $J_{\sigma_{1}}$-paths in $D_{A}$: $\pi_{1}=(5, 3, 4, 1, 2)$, $\pi_{2}=(5, 4, 1, 2)$, $\pi_{3}=(3, 4, 5, 1, 2)$, $\pi_{4}=(5, 4, 3, 2)$, $\pi_{5}=(4, 5, 3, 2)$, $\pi_{6}=(3, 4, 5, 2)$.




Applying Algorithm~\ref{algorithm1} to nonnegative elementary cycle $\sigma_{1}=(2, 2)$, we obtain the extremal generator {\color{blue} $e^{2}$}.

Applying Algorithm~\ref{algorithm2} to the maximum $J_{\sigma_{1}}$-path $\pi_{1}=(5, 3, 4, 1, 2)$ we first obtain
$v^{1}=e^{2}\oplus(A_{1}\otimes e^{2})\otimes
e^{1}=(1, 0, -\infty, -\infty, -\infty)^{T}.$
This is indeed an extremal, as it can be verified using \cite[Theorem 1]{Allamigeon2013} or \cite[Theorem 3.1]{SW-21}, and
the corresponding scaled extremal is {\color{blue} $(0, -1, -\infty, -\infty, -\infty)^{T}$}. The next generator is
$v^{4}=v^{1}\oplus(A_{4}\otimes v^{1})\otimes
e^{4}=(1, 0, -\infty, 2, -\infty)^{T},$ which is again an extremal, and the correponding scaled extremal is {\color{blue} $(-1, -2, -\infty, 0, -\infty)^{T}$}. The next generator is
$v^{3}=v^{4}\oplus(A_{3}\otimes v^{4})\otimes
e^{3}=(1, 0, 4, 2, -\infty)^{T}$, which is again an extremal,  and after scaling it becomes {\color{blue} $(-3, -4, 0, -2, -\infty)^{T}$}. Finally we compute
$v^{5}=v^{3}\oplus(A_{5}\otimes v^{3})\otimes
e^{5}=(1, 0, 4, 2, 3)^{T}$.  However, this is not an extremal and we turn to the next $J_{\sigma_{1}}$-path.

Applying Algorithm~\ref{algorithm2} to the maximum $J_{\sigma_{1}}$-path $\pi_{2}=(5, 4, 1, 2)$ we first obtain $v^{1}$ and $v^{4}$ that have been already computed. After this we compute new vector $v^{5}=v^{4}\oplus(A_{5}\otimes v^{4})\otimes
e^{5}=(1, 0, -\infty, 2, 3)^{T}$, which appears to be an extremal, and after scaling it becomes {\color{blue} $(-2, -3,  -\infty, -1, 0)^{T}$}.

Next we turn to the maximum $J_{\sigma_{1}}$-path $\pi_{3}=(3, 4, 5, 1, 2)$.
Note that $v^{1}$ has been obtained in the previous discussion of $J_{\sigma_{1}}$-path $\pi_{1}=(5, 3, 4, 1, 2)$. Then we compute the new vector
$v^{5}=v^{1}\oplus(A_{5}\otimes v^{1})\otimes
e^{5}=(1, 0, -\infty, -\infty, -1)^{T},$ which corresponds to the scaled extremal {\color{blue} $(0, -1,  -\infty, -\infty, -2)^{T}$}. As for the next generator
$v^{4}=v^{5}\oplus(A_{4}\otimes v^{5})\otimes
e^{4}=(1, 0, -\infty, 2, -1)^{T},$ it is not an extremal and we go to the next $J_{\sigma_1}$-path.

Next we turn to the maximum $J_{\sigma_{1}}$-path $\pi_{4}=(5, 4, 3, 2)$. We first obtain
$v^{3}=e^{2}\oplus(A_{3}\otimes e^{2})\otimes
e^{3}$, which is the scaled extremal {\color{blue} $(-\infty, 0, 0, -\infty, -\infty)^{T}$}. We next obtain
$v^{4}=v^{3}\oplus(A_{4}\otimes v^{3})\otimes
e^{4}=(-\infty, 0, 0, -5, -\infty)^{T}$, again an extremal, becoming {\color{blue} $(-\infty, 0, 0, -5, -\infty)^{T}$} after scaling.
Finally we compute
$v^{5}=v^{4}\oplus(A_{5}\otimes v^{4})\otimes
e^{5}=(-\infty, 0, 0, -5, -2)^{T}$, but this vector is not an extremal.

Working with the maximum $J_{\sigma_{1}}$-path $\pi_{5}=(4, 5, 3, 2)$, we first compute
$v^{3}$, but this is the same $v^3$ which we had for the previous $J_{\sigma_{1}}$-path $\pi_{4}=(5, 4, 3, 2)$. Then we compute
$v^{5}=v^{3}\oplus(A_{5}\otimes v^{3})\otimes
e^{5}=(-\infty, 0, 0, -\infty, -2)^{T}$, but this is not an extremal, so this path has not given us any new extremal generators and we have to turn to the
remaining path $\pi_6$.

Working with the maximum $J_{\sigma_{1}}$-path $\pi_{6}=(3, 4, 5, 2)$ we first obtain
$v^{5}=e^{2}\oplus(A_{5}\otimes e^{2})\otimes
e^{5}=(-\infty, 0, -\infty, -\infty, -2)^{T}$, which is an extremal and it becomes {\color{blue} $(-\infty, 0, -\infty, -\infty, -2)^{T}$} after scaling.
Then we compute  $v^{4}=v^{5}\oplus(A_{4}\otimes v^{5})\otimes
e^{4}=(-\infty, 0, -\infty, -9, -2)^{T}$, which is an extremal and it becomes {\color{blue} $(-\infty, 0, -\infty, -9, -2)^{T}$} after scaling. Then we compute
$v^{3}=v^{4}\oplus(A_{3}\otimes v^{4})\otimes
e^{3}=(-\infty, 0, 0, -9, -2)^{T}$, but this is not an extremal.

Thus we have exhausted the set of all maximal $J_{\sigma_1}$-paths leading to the loop $\sigma_1=(2,2)$. Next,
for the cycle $\sigma_2=(2,1,2)$ Algorithm~\ref{algorithm1} produces $v^{\sigma_2(1)}=(1, 0, -\infty, -\infty, -\infty)^{T}$ and $v^{\sigma_2(2)}=e^2$. Both of these extremals were generated above, and considering the maximal $J_{\sigma_2}$-paths cannot yield any new extremals.  For the cycle $\sigma_3=(2,3,2)$ observe that $v^{\sigma_3(2)}=e^2$ and $v^{\sigma_3(3)}=(-\infty, 0, 0, -\infty, -\infty)^{T}$, and we have the same observation as for $\sigma_2$.
For the cycle $\sigma_4=(2,3,4,1,2)$ we have $v^{\sigma_4(2)}=e^2$, $v^{\sigma_4(1)}=(1, 0, -\infty, -\infty, -\infty)^{T},$ $v^{\sigma_4(4)}=(1, 0, -\infty, 2, -\infty)^{T}$ and
$v^{\sigma_4(3)}=(1, 0, 4, 2, -\infty)^{T}$: all of them are extremals, but they were previously computed and considering the $J_{\sigma_4}$-paths cannot yield any new extremals in this case, either.}
\end{example}

Thus Algorithm \ref{algorithm3}, which combines tropical double description method and \cite[Algorithm 4.1]{wanghuili2020}, provides a new method to find the scaled basis for the solution space of the system \eqref{EqSupnolam} (or \eqref{Foralgori}) with $A\in\overline{\mathbb{R}}^{n\times n}$.
At least  theoretically, the computational complexity of Algorithm \ref{algorithm3} is still high since it is also based on all nonnegative elementary cycles and maximum $J_{\sigma}$-paths in $D_{A}$ (see \cite{wanghuili2020}). We see that (at least) all these cycles have to be examined, and that repeated extremals often appear in the process.  However, it is worth noting that Algorithm \ref{algorithm3} also has a mechanism of pruning the whole branches of generators that are not extremal and this mechanism may help to reduce the running time. It nevertheless remains a compelling challenge for the future work to avoid repeating extremals and further enhance the pruning mechanism of the algorithms presented here, in order to work out a more efficient method.

\end{document}